\documentclass[10pt]{amsart}

\usepackage{graphicx}
\usepackage{amsfonts}
\usepackage{amsmath}
\usepackage{amssymb}
\usepackage{amsopn}
\usepackage{mathrsfs}
\usepackage{enumerate}

\usepackage{graphicx}
\usepackage{amsfonts}
\usepackage{amsmath}
\usepackage{amssymb}
\usepackage{amsopn}
\usepackage{mathrsfs}
\usepackage{soul}
\usepackage{enumerate}
\usepackage{wasysym}

\DeclareMathOperator*{\ddim}{dim}
\DeclareMathOperator*{\Var}{Var}
\DeclareMathOperator*{\support}{supp}
\DeclareMathOperator*{\Card}{Card}

\DeclareMathOperator{\arcos}{arcos}

\DeclareMathOperator*{\graph}{Graph}
\DeclareMathOperator*{\Tan}{Tan}
\DeclareMathOperator*{\Nor}{Nor}
\DeclareMathOperator*{\Reach}{Reach}
\DeclareMathOperator*{\epi}{epi}
\DeclareMathOperator*{\Lip}{Lip}
\DeclareMathOperator*{\Jac}{Jac}
\DeclareMathOperator*{\pr}{proj}
\DeclareMathOperator*{\Range}{Range}
\DeclareMathOperator*{\dist}{dist}

\newcommand{\C}{\mathbb{C}}
\newcommand{\R}{\mathbb{R}}
\newcommand{\N}{\mathbb{N}}

\newcommand{\Graph}[1]{\graph(#1)}
\renewcommand{\phi}{\varphi}
\renewcommand{\wp}{\mathscr P}

\newtheorem{theorem}{Theorem}[section]
\newtheorem{cor}{Corollary}[theorem]
\newtheorem{proposition}{Proposition}[section]

\newtheorem{definition}[theorem]{Definition}

\newtheorem{remark}[theorem]{Remark}

\author{Federico Piazzon}
\address{room 712 Department of Mathematics, Universit\'a di Padova, Italy. Phone +39 0498271260}
\email{\underline{fpiazzon@math.unipd.it}} 
\urladdr{http://www.math.unipd.it/~fpiazzon/   (work in progress)}
\thanks{Supported by INdAM GNCS and Doctoral School on Mathematical Sciences Univ. of Padua}
\date{\today}
\title[Optimal AM]{Optimal Polynomial Admissible Meshes on Some Classes of Compact Subsets of $\R^d$}

\begin{document}
\begin{abstract}
We show that any compact subset of $\R^d$ which is the closure of a bounded star-shaped Lipschitz domain $\Omega$, such that $\complement \Omega$ has positive reach in the sense of Federer, admits an \emph{optimal AM} (admissible mesh), that is a sequence of
polynomial norming sets with optimal cardinality. This extends a recent result of A. Kro\'o on $\mathscr C^ 2$ star-shaped domains.

Moreover, we prove constructively the existence of an optimal AM for
any $K := \overline\Omega \subset \R^ d$ where $\Omega$ is a bounded $\mathscr C^{ 1,1}$ domain. This is done by a particular multivariate sharp version of the Bernstein Inequality via the distance function.
\end{abstract}
\maketitle
\tableofcontents
\section{Introduction}
Let us denote by $\wp^n(\R^d)$ the space of polynomials of d real variables
having degree at most $n$. We recall that a compact set $K \subset\R^d$ is said to be polynomial determining if any polynomial vanishing on $K$ is necessarily the
null polynomial.

Let us consider a polynomial determining compact set $K \subset\R^d$ and let
$A_n$ be a subset of $K$. If there exists a positive constant $C_n$ such that for any
polynomial $p \in \wp^n(\R^d)$ the following inequality holds
\begin{equation}
\|p\|_K\leq C_n \|p\|_{A_n},\label{amdef}
\end{equation}
then $A_n$ is said to be a \emph{norming set of constant} $C_n$ \emph{for} $\wp^n (\R^d ).$ Here and throughout the paper we use this notation: $\|f\|_X := \sup_{x\in X}|f (x)|$ for any bounded function on $X.$

Let $\{A_n \}$ be a sequence of norming sets for $\wp^n (\R^d )$ with constants $\{C_n \}$, and suppose that both $C_n$ and $\Card(A_n)$ grow at most polynomially with $n$ (i.e., $\max\{C_n , \Card(A_n )\} = O(n^s)$ for a suitable $s \in \N$), then $\{A_n\}$ is said to be a \emph{weakly admissible mesh} (WAM) for K; see\footnote{The original definition in \cite{CL08} is actually a little weaker (sub-exponential growth instead of polynomial growth is allowed), here we prefer to use the present one which is the most common in the literature.} \cite{CL08}. If $C_n \leq C$ $\forall n$, then $\{A_n \}_\N$ is said an \emph{admissible mesh} (AM) for $K$; in the sequel, with a little abuse of notation, we term (weakly) admissible mesh not only the whole sequence but also its $n$-th element $A_n$.

Observe that necessarily
$$\Card A_n\geq N=\ddim\wp^n(\R^d)={{n+d}\choose{n}}=O(n^d)$$
since a (W)AM is $\wp^n(\R^d )$-determining, i.e., any polynomial in $\wp^n(\R^d )$ vanishing on $A_n$ is the zero polynomial. When $\Card(A_n) =O(n^d)$, following Kro\'o \cite{KRRE11}, we speak of an optimal admissible mesh.

We recall that AMs are preserved by affine transformations and can be
constructed incrementally by finite union and product. Moreover they are
stable under small perturbations and smooth mappings; see \cite{PV10} and \cite{PV11}. For a survey on WAMs properties and applications we refer to \cite{BDSV10}.

The study of AMs has several computational motivations. Indeed, it has
been proved by Calvi and Levenberg that discrete least squares polynomial
approximations based on (W)AMs are nearly optimal in the uniform norm,
see \cite[Thm. 1]{CL08}. Moreover, discrete extremal sets extracted from (W)AMs (see for instance \cite{BDSV10},\cite{BV11}), are known to be good interpolation sets and to behave asymptotically like Fekete points, namely the corresponding sequences of uniform probability measures converge weak$^{\ast}$ to the pluripotential equilibrium measure of the underlying compact set; see \cite{BB08} \cite{BB11} or the survey \cite{LevBB10}.
We recall that it is possible to construct an admissible mesh with $O(n^{ rd})$ points on any real compact set satisfying a Markov Inequality \cite{BOMI93} with exponent $r.$ The mesh can be obtained by intersecting the compact set with a uniform grid having $O(n^{-r} )$ step size by \cite[Thm. 5]{CL08}.

Indeed, the hypotheses of \cite[Thm. 5]{CL08} are not too restrictive. For instance
one has a Markov Inequality with exponent $2$ for any compact set $K \subset
\R^d$ satisfying a uniform cone condition \cite{Bau}. Thus also for the closure of any bounded Lipschitz domain. However a Markov Inequality holds with an
exponent possibly greater than 2 even for more general classes of sets; see
\cite{PAPLE86} and \cite{PAPLE86b} for details.

The cardinality growth order of AMs built by this procedure, however, causes severe computational drawbacks already for $d=2.$ This gives a strong practical motivation to construct low-cardinality admissible meshes, in particular optimal ones.

It has been proved in \cite{LevSur} that for any compact polynomial determining $K \subset \C^d$ there exists an admissible mesh with $O((n \log n)^ d )$ cardinality, unfortunately the method relies on the determination of Fekete points, which are not known in general and whose construction is an extremely hard task.

In order to build meshes with nearly optimal cardinality growth order one can restrict his attention to sets with simple geometry such as simplices, squares, balls and their images under any polynomial map (see for instance \cite{BCLSV10}) or can look at some specific geometric-analytic classes of sets; the present paper follows the latter idea.

In \cite{KR13} the author proves that any compact star-shaped set $K \subset \R^d$ with Minkowski Functional (see for instance \cite[pg. 6]{Br}) having $\alpha$-Lipschitz gradient has an admissible mesh $\{Y_n\}$ with
$$\Card Y_n = O(n^\frac{2d+\alpha-1}{\alpha+1}).$$

In particular he notices that this implies the existence of optimal AMs
for the closure of any $\mathscr C^ 2$ star-shaped bounded domain.

While writing this paper we received a new preprint (now
published) by A. Kro\'o where the author improves his estimate above by a fine use of Minkowski Functional smoothness; \cite[Theorem 3]{KR13}.

In \cite{KRRE11} he also conjectured that any real convex body has an optimal admissible mesh. In this work we build such optimal admissible meshes on two relevant classes of compact sets.

The paper is organized as follows.

In \textbf{Section 2} we work on star-shaped compact sets in $\R^d$ with nearly minimal boundary regularity assumptions. We prove in Theorem 2.3 that if $\Omega\subset \R^d$ is a bounded star-shaped Lipschitz domain such that $\complement\Omega$ has positive reach (see Definition \ref{positivereach}), then $K := \overline\Omega$ has an optimal admissible mesh.

In \textbf{Section 3} we address the same problem but we drop the star-shape
assumption on $K$, it turns out that a little stronger boundary regularity is
needed. In Theorem 3.6 we prove that if $\Omega$ is a bounded $\mathscr C^{1,1}$ domain of $\R^d$, then there exists an optimal admissible mesh for $K := \overline\Omega$.

In the \textbf{Appendices} we provide, for the reader’s convenience, a quick review of some definitions and results from non-smooth and geometric analysis and geometric measure theory that are involved in the framework of this paper.


\section{Optimal AMs for star-shaped sets having complement with positive reach}
In approximation theory it is customary to consider as mesh parameter the \textbf{fill distance} $h(Y)$ of a given finite set of points $Y$ with respect to a compact subset $X$ of $\R^d$.
\begin{equation}
h(Y):=\sup_{x\in X}\inf_{y\in Y}|x-y|.\label{clfilldistance}
\end{equation}
In this definition it is not important whether the segment $[x,y]$ lies in $X$ or not. If one wants to control the minimum length of paths joining $x$ to $y$ and supported in $X$ then one may consider the following straightforward extension of the concept of fill distance given above.
\begin{definition}[Geodesic Fill-Distance]
Let $Y$ be a finite subset of the set $X\subset \R^d$, then we set
\[
\mathscr A_{x,y}(X):=\left\{\gamma\in \mathscr C([0,1],X):\gamma(0)=x,\gamma(1)=y,\Var[\gamma]<\infty\right\}\]
and define
\begin{equation}
h_X(Y):=\sup_{x\in X}\inf_{y\in Y}\inf_{\gamma\in \mathscr A_{x,y}} \Var[\gamma],
\label{geofill}
\end{equation}
the geodesic fill distance of $Y$ over $X$.
\end{definition}
Here and throughout the paper we denote by $\Var[\gamma]$ the total variation of the curve $\gamma$, $$\Var[\gamma]:=\sup_{N\in \N}\,\sup_{0=t_0<t_1\dots<t_N=1}\sum_{i=1}^N|\gamma(t_i)-\gamma(t_{i-1})|.$$
Notice that if we make the further assumption of the local completeness of $X,$ then it ensures the existence of a length minimizer in $\mathscr A_{x,y}(X)$ provided it is not empty, that is if there exists a rectifiable curve $\psi$ connecting any $x$ and $y$ in $X$ such that $ \Var[\psi]\leq L<\infty$. Thus if $X$ has finite geodesic diameter, which will be the case of all instances considered in this paper, then we can replace $\inf_{\gamma\in \mathscr A_{x,y}} \Var[\gamma]$ by $\min_{\gamma\in \mathscr A_{x,y}} \Var[\gamma]$ in \eqref{geofill}.

Now we want to build a mesh on the boundary of a bounded Lipschitz domain having a given geodesic fill distance but keeping as small as possible the cardinality of the mesh. Then we use such a ``geodesic" mesh to build an optimal AM for the closure of the domain.

For the reader's convenience we recall here that a domain $\Omega\subset \R^d$ is termed a (uniformly) \emph{Lipschitz domain} if there exist $0<L<\infty,$ $r>0$ and an open neighbourhood $B$ of $0$ in $\R^{d-1}$ such that for any $x\in \partial \Omega$ there exists $\phi_x:B\rightarrow ]-r,r[$ and a rotation $R_x\in SO_d$ such that $\phi_x(0)=0,$ $\Lip(\phi_x)\leq L$ and 
\small
$$R_x^{-1}\left(\Omega \cap (x+R_x(B\times ]-r,r[))-x\right)=\epi \phi_x:=\{(\xi,t):\xi\in B,t\in ]-R,\phi_x(t)[\}.$$   
\normalsize
The following result, despite its rather easy proof, is a key element in our construction. For a bounded Lipschitz domain the euclidean and geodesic (on the boundary) distances restricted to the boundary are equivalent. 

\begin{proposition}
Let $\Omega$ be a bounded Lipschitz domain in $\R^d$, then there exists $\bar h>0$ such that there exists $Y_h\subset X := \partial\Omega,0<h<\overline h$ and the following hold:
\begin{enumerate}[(i)]
 \item $\Card Y_h =O \left(h^{1-d}\right)$ as $h\rightarrow 0$.
 \item $h_{X}(Y_h)\leq h$.
\end{enumerate}
\label{boundarymesh}
\end{proposition}
\begin{figure}
\centering
\includegraphics[width=5in]{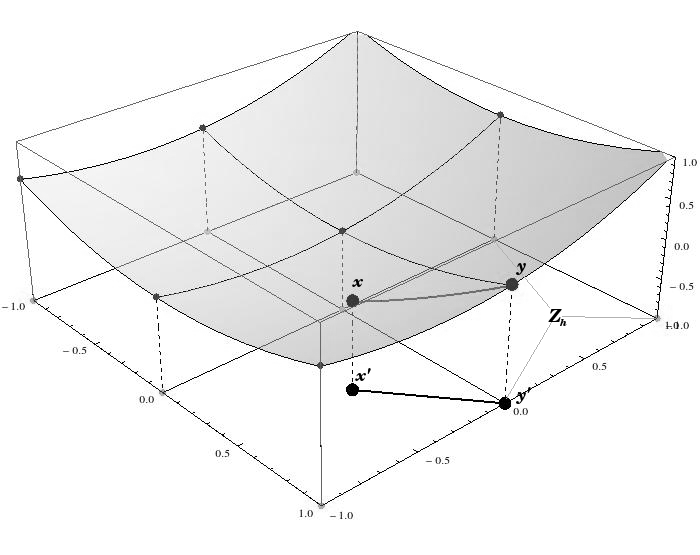}
\caption{The \emph{geodesic mesh} in the proof of Proposition \ref{boundarymesh} is built by lifting the grid mesh $Z_h$ by the local parametrization of the boundary $X$. The curve $\gamma_x$ connecting $x$ to $y$ is similarly produced by lifting the segment $[x',y'].$}
\label{costruzione}
\end{figure}

\begin{proof}
Here we denote by $B_\infty^s(x_0,r)$ the $s$ dimensional ball of radius $r$ centered at $x_0$ with respect to the norm $|x|_\infty:=\max_{i\in\{1,2,\dots,s\}}|x_i|,$ i.e. the coordinate cube centered at $x_0$ and having sides of length $2r.$

Since $\Omega$ is a Lipschitz domain  using the above notation we can write

$$\left(\,x+R_x B^d_\infty(0,r)\, \right)\cap \partial\Omega= R_x\graph(\phi_x).$$

Let us denote the \emph{graph} function of $\phi_x$ by $g_x:B^{d-1}_\infty(0,r)\longrightarrow\R^d,$ that is $B^{d-1}_\infty(0,r)\ni \boldsymbol\xi\mapsto\{\xi_1,\xi_2,\dots,\xi_{d-1},\phi_x(\boldsymbol\xi)\}=g_x(\boldsymbol\xi).$

By compactness we can pick $x_1,x_2,\dots,x_{M(r)}\in\partial \Omega$ such that
$$\partial \Omega\subseteq\cup_{i=1}^{M(r)}X_i=: \cup_{i=1}^{M(r)}\left(\,\left(x_i+R_{x_i} B^d_\infty(0,r)\right)\cap\partial\Omega \,\right).$$
Let $\bar h:=r\sqrt{1+L^2}$, take any $0<h\leq \bar h$ and let us consider the grid of step-size $\frac{h}{\sqrt{(1+L^2)}}$  in the $d-1$ dimensional cube
$$Z_h:=\left(\left\lbrace-r+\frac{j h}{\sqrt{(1+L^2)}} \right\rbrace_{j=0,1,\dots, \lceil\frac{2r\sqrt{1+L^2}}{h}\rceil}\right)^{d-1}\subset B^{d-1}_\infty(0,r) ,$$
where $\lceil\cdot\rceil$ is the ceil operator.
Set
\begin{eqnarray*}
Y_h^i &:=&x_i+R_{x_i}\left(g_{x_i}(Z_h)\right),\\
Y_h&:=&\cup_{i=1}^{M(r)}Y_h^i.
\end{eqnarray*}
Now notice that
\begin{eqnarray*}
\Card{Y_h}&\leq& \sum_{i=1}^{M(r)}\Card Y_h^i=M(r)\Card Z_h\\
&=&M(r)\left(1+\left\lceil\frac{2r\sqrt{1+L^2}}{h}\right\rceil\right)^{d-1}\\
&=&O(h^{1-d}).
\end{eqnarray*}
  
In order to verify the \emph{(ii)} for any $x\in \partial \Omega$ we explicitly find $y\in Y_h$ and build a curve $\gamma_x$ connecting $x$ to $y$ whose variation gives an upper bound for the geodesic distance of $x$ from $Y_h$. For the following construction we refer to the Figure \ref{costruzione}.

Take any $x\in \partial \Omega,$ then there exist (at least one) $i\in \{1,2,\dots,M(r)\}$ such that $x\in X_i.$ Let us pick such an $i.$

Let us denote by $\pr_i$ the canonical projection on the first $d-1$ coordinates acting from $R_{x_i}^{-1}\,\left(\left(x_i+R_{x_i} B^d_\infty(0,r)\right)\cap \partial\Omega\;-\;x_i\;\right)  $ onto $B^{d-1}_\infty(0,r).$

Let $x':=\pr_i(x)$, by the very construction we can find $y'\in Z_h$ such that $|x'-y'|\leq\frac{h}{\sqrt{1+L^2}}=:h',$ moreover the whole segment $[x',y']$ lies in $B_\infty^{d-1}(0,r).$

We consider the curve  $\alpha_x:\xi\mapsto x'+\xi \frac{y'-x'}{|y'-x'|},\xi\in[0,h']$ and we set $\gamma_x(\xi):=x_i+g_{x_i}(\alpha(\xi))$ the curve that joins x to $y:=x_i+g_{x_i}(y')\in Y_h$ obtained by mapping the segment $[x',y']$ under $g_{x_i}.$

Now we use \emph{Area Formula} \cite{Fed69}  \cite{Eva}[Th. 1 pg. 96] to compute the length of the Lipschitz curve $\gamma_x.$
\footnotesize
\begin{eqnarray}
 & &\Var[\gamma_x]=\int_0^{h'}\Jac[\gamma](t)dt=\label{area}\\
&=&\int_0^{h'}\left[ \sum_{i=1}^{d-1} \left(\frac{y'_i-x'_i}{|y'-x'|}\right)^2\;+\; \dots +\left(\nabla\phi_x\left(x'+t \frac{y'_i-x'_i}{|y'-x'|}\right)\cdot(t \frac{y'_i-x'_i}{|y'-x'|})\right)^2  \right]^\frac 1 2dt\nonumber\\
&=&\int_0^{h'}\left[\left| \frac{y'-x'}{|y'-x'|}\right|^2+ L \left| \frac{y'-x'}{|y'-x'|}\right|^2\right]^\frac 1 2 dt\leq \sqrt{1+L^2}h'= h\nonumber.
\end{eqnarray}
\normalsize
Here $\Jac$ is the Jacobian of a Lipschitz mapping, see \cite{Eva}[pg. 101].

We take the maximum over $x\in \partial \Omega$ using \eqref{geofill}, notice that our $\gamma_x$ by the construction is an element of $\mathscr A_{x,y},$
\begin{equation*}
 h_{\partial\Omega}(Y_h)=\sup_{x\in X}\inf_{y\in Y_h}\inf_{\eta\in \mathscr A_{x,y}} \Var[\eta]\leq\sup_{x\in X}\Var[\gamma_x]\leq h.
\end{equation*}
\Square\end{proof}

Now we are ready to state and prove the main result of this section. We build an optimal mesh for a star shaped Lipschitz bounded domain having complement of positive reach by the following technique. First, we consider the hypersurfaces given by the images of the boundary of the domain under a one parameter family of homotheties, being the parameter chosen as Chebyshev points scaled to the suitable interval. We prove that this family of hypersurfaces is a norming set for the given compact. The second key element is that on each such hypersurface we can use a Markov Tangential Inequality with minimal (with respect to the degree of the considered polynomial) growth rate $n$.   

\begin{theorem}
Let $\Omega\subset \R^d$ be a bounded star-shaped Lipschitz domain such that $\complement\Omega$ has positive reach (see Definition \ref{positivereach}),
then $K:=\overline{\Omega}$ has an optimal polynomial admissible mesh.
\label{mainresult}\end{theorem}
\begin{proof}
We can suppose without loss of generality the center of the star to be $0$ by stability of AM under euclidean isometries \cite{BDSV10}.

Let us set $b_n^i(r):=\frac{r}{2}(1+\cos{\frac{\pi(2n-i)}{2n}})$ for any $r>0$ $i=1,2,\dots 2n+1$. 
By a well known result (\cite{EZ64}) the set $G_n(r)$ of all $b_n^i(r)$'s (varying the index $i$) is an admissible mesh of degree $n$ and constant $\sqrt 2$ for the interval $[0,r]$:
\begin{equation} 
\|p\|_{[0,r]}\leq \sqrt{2} \|p\|_{G_n(r)}\;\;\forall p\in \mathscr P^n.
\label{intervalmesh}
\end{equation}

Let us take any $x\in X:=\partial K$ and consider the set $\tilde G_n(x):=xG_n(1)$, notice that $\tilde G_n(x)\subset K$ because $K$ is star-shaped.

One can set $Z_n:=\cup_{x\in X} \tilde G_n(x)$ , i.e., $Z_n$ is the union of the images of $X$ under the homotheties having parameters $\cos{\frac{\pi(2n-i)}{2n}}.$ See Figure \ref{butterfly}.

\begin{figure}
\centering
\includegraphics[width=2in]{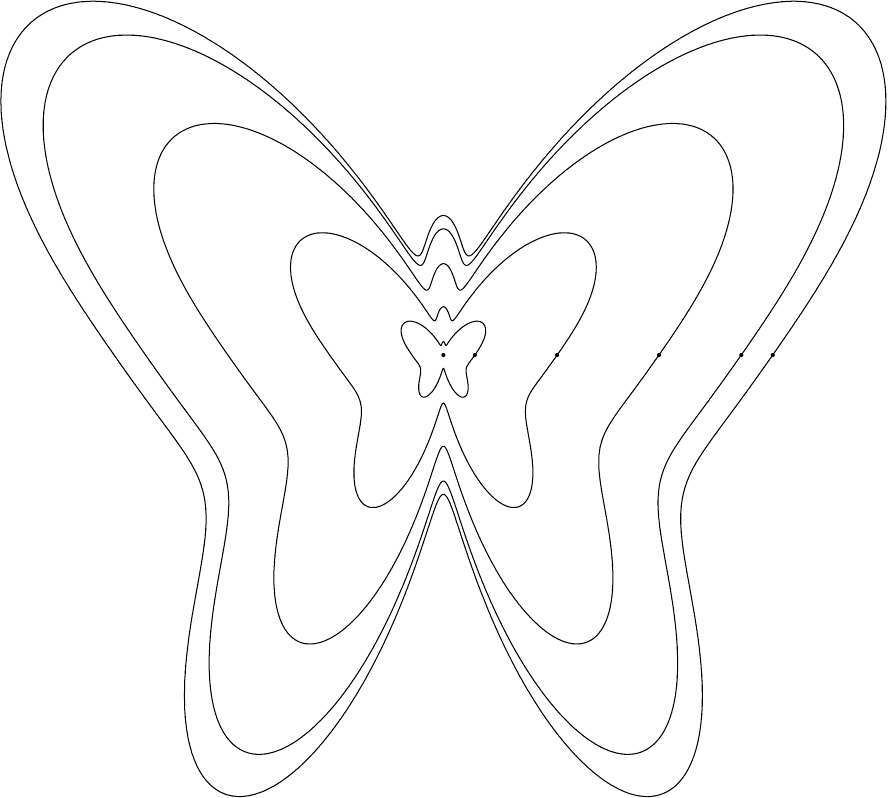}
\caption{The geometry of $Z_n$.}
\label{butterfly}
\end{figure}

Notice that the restriction of any polynomial of degree at most $n$ in $d$ variables to any segment is a univariate polynomial of degree at most $n$, then due to \eqref{intervalmesh} $Z_n$ are norming sets for $K$, that is
\begin{equation}
\|p\|_K\leq \sqrt{2} \|p\|_{Z_n}\;\;\forall p\in \mathscr P^n(\R^d)\label{norming}.
\end{equation}
Therefore we are reduced to finding an admissible polynomial mesh of degree $n$ for $Z_n$.

Let us consider any\footnote{Notice that $X$ is compact connected, nonempty and consists of an infinite number of points, obviously it contains an infinite number of Lipschitz curves.} Lipschitz curve $\gamma:[0,1]\rightarrow X$, by Proposition \ref{reachball}  for a.e. $s\in]0,1[$ there exists $v\in\mathbb S^{d}$ such that
\begin{enumerate}
 \item $B(\gamma(s)+rv,r)\subseteq K$ and
\item $ \gamma'(s)\in \mathcal T_{\gamma(s)} \partial B(\gamma(s)+rv,r).$ 
\end{enumerate}
Hereafter $\mathcal T_p M$ is, as customary, the tangent space to $M$ at $p\in M.$

Since the boundary of the ball is a compact algebraic manifold, it admits \emph{Markov Tangential Inequality} of degree $1$  (see \cite{BLMT} and the references therein), moreover the constant of such an inequality is the inverse of the radius of the ball:
\begin{equation}
\left|\frac{\partial p}{\partial v}(x)\right|\leq \frac{|v|}{r} n \|p\|_{B(x_0,r)}  \;\;\forall p\in \mathscr P^n(\R^d)\;  ,\,\forall v\in \mathcal T_x {\partial B(x_0,r)}.\label{MarkovBall} 
\end{equation}

Let us recall (see for instance \cite{AGS08}[Lemma 1.1.4]) that any Lipschitz curve $\gamma$ can be re-parametrized by arclength by the inversion of $t\mapsto \Var[\gamma|_{[0,t]}]$, obtaining a Lipschitz curve
\begin{eqnarray*}
\tilde\gamma:[0,\Var[\gamma]]&\rightarrow & X\\
\Var[\tilde\gamma]&=& \Var[\gamma]\\
\Lip[\tilde\gamma]&=& 1 =_{\text{a.e.}}|\tilde\gamma'|
\end{eqnarray*}

Therefore (using Rademacher Theorem, see for instance \cite{Eva}[Th.2 pg 81]) for a.e. $s\in ]0,1[$ we have
\begin{eqnarray}
\left| \frac{\partial(p\circ\tilde\gamma)}{\partial t}(t)\right|&=& |\nabla p(\tilde\gamma(t))\cdot\tilde\gamma'(t)|\\
&\leq& \frac{|\tilde\gamma'(t)|n}{r}\|p\|_{B(\tilde\gamma(t)+r v,r)}\leq\frac{n}{r}\|p\|_{K}.\label{estimate}
\end{eqnarray}

By Proposition \ref{boundarymesh} we can pick subsets $Y_{\frac{r}{2n}}$ on $X$ such that $h_X\left(Y_{\frac{r}{2n}}\right)\leq \frac{r}{2n}$ and $\Card Y_{\frac{r}{2n}}=O(n^{d-1})$. For notational convenience we write $Y_n$ in place of $Y_{\frac{r}{2n}}.$

Let us now pick any $x\in X$ and consider $\gamma$, an arc connecting a closest point $y_n^i$ of $Y_n$ to $x$ and $x$ itself such that $\Var[\gamma]\leq \frac{r}{2n}$, parametrized in the arclength.

By the Lebesgue Fundamental Theorem of Calculus for any $p\in \mathscr P^n(\R^d)$ one has
\begin{eqnarray*}
|p(x)|&\leq& |p(y_n^i)|+\left|\int_0^{\Var[\gamma]} \frac{\partial (p\circ \gamma)}{\partial \xi}(\xi) d\xi\right|\\
&\leq&|p(y_n^i)|+\int_0^{\Var[\gamma]}\left| \nabla p(\gamma(\xi))\cdot\gamma'(\xi) \right|d\xi\\
&\leq&|p(y_n^i)|+\int_0^{r/2n}\frac{n}{r}\|p\|_Kd\xi\leq |p(y_n^i)|+ \frac{1}{2}\|p\|_K
\end{eqnarray*}
where in the last line we used \eqref{estimate}. Thus we have
\begin{equation}
\|p\|_X\leq \|p\|_{Y_n}+\frac{1}{2} \|p\|_K .
\label{estimate2}
\end{equation}
By the properties of rescaling, setting $b_n^i:=b_n^i(1)=\frac{1+\cos{(i\pi/n)}}{2},$ we have also
\begin{equation*}
\|p\|_{b_n^i X}\leq \|p\|_{b_n^i Y_n}+1/2 \|p\|_{b_n^i K}\leq \|p\|_{b_n^i Y_n}+\frac{1}{2}\|p\|_K,
\end{equation*}
for, consider the homothety $\Theta_n^i:\R^d\rightarrow \R^d$, where $\Theta_n^i(x):=\frac x {b_n^i}$ and write the inequality \eqref{estimate2} for each $q_{i,n}:=p\circ\Theta_n^i.$

Therefore, taking the union over $i=0,1,2n$ and using $x\tilde G_n=\cup_{i=0}^{m_n} b_n^ix$ and $Z_n=\cup_{x\in X}x\tilde G_n$, we have
\begin{equation*}
 \|p\|_{Z_n}=\|p\|_{\cup_{x\in X}(\cup_i b_n^ix)}\leq\|p\|_{\cup_i b_n^i Y_n}+\frac{1}{2}\|p\|_K.
\end{equation*}

Hence, setting $X_n:=\cup_{i=0}^{2n}b_n^iY_n,$ we can write
\begin{equation*}
\|p\|_{Z_n}\leq \|p\|_{X_n} +\frac{1}{2} \|p\|_K.
\end{equation*}
Now we can use \eqref{norming} to get $\|p\|_K\leq \sqrt 2\left( \|p\|_{X_n} +\frac{1}{2} \|p\|_K\right)$ and hence
\begin{equation*}
\|p\|_K\leq \frac{2\sqrt 2}{2-\sqrt{2}}\|p\|_{X_n}=2(\sqrt{2}+1)\|p\|_{X_n}.
\end{equation*}
Thus $X_n$ is an admissible polynomial mesh for $K$.
The set $X_n$ is the disjoint union of $2n+1$ sets $b_n^i Y_n$,thus $$\Card X_n=(2n+1)O(n^{d-1})=O(n^d),$$
therefore $X_n$ is an optimal admissible mesh of constant $2(\sqrt{2}+1)$.
\Square\end{proof}

This result should be compared to the recent article \cite[Theorem 3]{KR13}. The results achieved by the author are set in a  little more general context, still they do not cover the case of a Lipschitz domain with complement having Positive Reach but not being $\mathscr C^{1,1-2/d}\,, d\geq 2$ globally smooth. The key element here is that inward pointing corners and cusps are allowed in our setting, while they are not in \cite{KR13}.

From an algorithmic point of view an AM built by a straightforward application of Theorem \ref{mainresult} may be refined. Informally speaking such a collocation technique creates AMs that are clustered near the center of the star, while this seem to have no geometrical nor analytical meaning.

This issue can be partially removed by some minor modifications of the construction which turn the proof of Theorem \ref{mainresult} in a more efficient algorithm. 

Theorem \ref{mainresult} is formulated in a rather general way, here we provide two corollaries that specialize such result.

It has been shown (see \cite{ABMM11}) that $\mathscr C^{1,1}$ domains (see \ref{c11domain}) of $\R^d$ are characterized by the so called \emph{uniform double sided ball condition}, that is, $\Omega$ is a $\mathscr C^{1,1}$ domain iff there exists $r>0$ such that for any $x\in \partial \Omega$ there exist $v\in \mathbb S^{d-1}$ such that we have $B(x+r v,r)\subseteq \Omega$ and $B(x-r v,r)\subseteq \complement\overline\Omega$, this property in particular says that $\complement \Omega$ (and $\Omega$ itself) has positive reach \ref{positivereach}. Therefore the following is a straightforward corollary of our main result.

\begin{cor}
Let $\Omega$ be a bounded star-shaped $\mathscr C^{1,1}$ domain, then its closure has an optimal AM.\label{C11case}
\end{cor}

It is worth recalling that such domains can also be characterized by the behavior of the oriented distance function of the boundary (i.e. $b_\Omega(x):=d(x,\Omega)-d(x,\complement\Omega)$). For any such $\mathscr C^{1,1}$ domain there exists a (double sided) tubular neighborhood of the boundary where the oriented distance function has the same regularity of the boundary, this condition characterizes $\mathscr C^{1,1}$ domains too.
This framework is widely studied in \cite{Del} and \cite{zodeldist}.

In the planar case a similar result holds under slightly weaker assumptions.

\begin{theorem}[\cite{PV13}]\label{mainresultcoro}
 Let $\Omega$ be a bounded star-shaped domain in $\R^2$ satisfying a \emph{Uniform Interior Ball Condition} UIBC (see Definition \ref{uniforminteriorball}),
then $K:=\overline{\Omega}$ has an optimal polynomial admissible mesh.
\end{theorem}

A comparison of the statements of Theorem \ref{mainresult} and Theorem \ref{mainresultcoro} reveals that actually in the second one we are dropping two assumptions, first the domain is no longer required to be Lipschitz, second we ask the weaker condition UIBC instead of complement of positive reach.

The first property is assumed to hold in the proof of the general case to make possible the construction of the geodesic mesh with a control on the asymptotics of the cardinality. In $\R^2$ the boundary of a bounded domain satisfying the UIBC is rectifiable; see \cite{FuGePi12}. Therefore, the geodesic mesh can be created by equally spaced (with respect to arc-length) points. 

On the other hand the role of the second missing property is recovered by a deep fact in measure theory. If a set has the UIBC then then the set of points where the normal space (see Definition \ref{tangentnormal}) has dimension greater or equal to $k$ has locally finite ${d-k}$ Hausdorff measure; \cite{Fed59, MKV12}. In our bi-dimensional (i.e., $d=2$) case this result reads as follow: the normal space has dimension greater or equal to $k=2$ on a subset having $0-$Hausdorff measure equal to $0,$ that is a finite set \cite{Fed59}. Moreover it can be proved that, apart from this small set, the \emph{single valued} normal space is Lipschitz.


\section{Optimal AM for $\mathscr C^{1,1}$ domains by distance function }

As we mentioned above, in \cite{KRRE11} the author conjectures that any real compact set admits an optimal AM, in this section we prove (in Theorem \ref{mainresult2}) that this holds at least for any real compact set $K$ which is the the closure of a bounded $\mathscr C^{1,1}$ domain $\Omega,$ see \ref{c11domain}

We denote by $d_{\complement\Omega}(\cdot)$ the distance function w.r.t. the complement $\complement \Omega$ of $\Omega$, i.e. 
\begin{equation}
d_{\complement\Omega}(x):=\inf_{y\in\complement\Omega}|y-x|,\label{distdef}
\end{equation}
 and by $ \pr_{\complement\Omega} (\cdot)$ the \emph{metric projection} onto $\complement \Omega$ i.e., $ \pr_{\complement\Omega} (x)$ is the set of all minimizer of \eqref{distdef}. 
We continue to use the same notation as in the previous section for the closure and the boundary of $\Omega$, namely
$ X:=\partial \Omega$ and $K:=\overline \Omega.$

Let us give a sketch of the overall geometric construction before giving details.

First for a given $\mathscr C^{1,1}$ domain $\Omega$ we take $0<\delta<2r_\Omega,$ where $r_\Omega$ is the maximum radius of the ball of the uniform interior ball condition satisfied by $\Omega.$ 

We can split $K:=\overline\Omega$ as follows
\begin{eqnarray*}
\overline \Omega&=&K_\delta\cup \overline{\Omega^\delta}\;\text{ where}\\
K_\delta &:=&\{x\in \Omega:d_{\complement\Omega}(x)\leq\delta\}\;\text{and}\\
\Omega^\delta&=&\Omega\setminus K_\delta.
\end{eqnarray*}

To construct an AM of degree $n$ on $\overline \Omega$ we work separately on $K_\delta$ and $\overline{\Omega^\delta}$ to obtain inequalities of the type
\begin{eqnarray*}
\|p\|_{K_\delta}&\leq &\|p\|_{Z_{n,\delta}}+ \frac{1}{\lambda} \|p\|_{K}\;,\lambda>1\text{  and}\\
\|p\|_{\Omega^\delta}&\leq &2\|p\|_{Y_{n,\delta}}+ \frac{2}{\mu} \|p\|_{K}\;,\mu>1,
\end{eqnarray*}
for $p\in \wp^n,$ where $Z_{n,\delta}\subset K_\delta $ and $Y_{n,\delta}\subset \Omega^\delta$ are suitably chosen finite sets.

In the case of $K_\delta$ this is achieved by the trivial observation $x\in K_\delta$ implies $\overline{B(x,\delta)}\subseteq \overline \Omega$ and therefore one can bound any directional derivative of a given polynomial using the univariate Bernstein Inequality (see Theorem \ref{bernsteinth} below). The resulting inequality is a variant of a Markov Inequality with exponent $1$ which is convenient and allow us to build a low cardinality mesh by a modification of the reasoning in \cite{CL08}.

The construction of an AM on  $\overline{\Omega^\delta}$ is more complicated. The resulting mesh is given by points lining on some properly chosen level surfaces of $d_{\complement\Omega}.$ The result is proved using the regularity property of the function  $d_{\complement\Omega}$ in a small tubular neighborhood of $X$ and the Markov Tangential Inequality for the sphere.

\subsection{Bernstein-like Inequalities and polynomial estimates via the distance function.}
For the reader's convenience we recall here the Bernstein Inequality.
\begin{theorem}[Bernstein Inequality]\label{bernsteinth}
 Let $p\in\mathscr{P}^n(\R)$, then for any $a<b\in \R$ we have
\begin{equation}
 |p'(x)|\leq\frac{n}{\sqrt{(x-a)(b-x)}}\|p\|_{[a,b]},\;x\in ]a,b[.
\end{equation}
\end{theorem}

Let us introduce the following notation illustrated in Figure \ref{elle}.
\begin{eqnarray}
 l(x)&:=&\min_{y\in \pr_{\complement\Omega} (x)}\inf\left\lbrace\lambda>0:y+\lambda\frac{x-y}{|x-y|}\notin \Omega\right\rbrace\;\;x\in \Omega\\
l_\Omega&:=&\inf_{x\in\Omega}l(x).
\end{eqnarray}
\begin{remark}
 In the case when $\Omega$ is a $\mathscr C^{1,1}$ domain one has the estimate $l_\Omega\geq 2r$ where $r<\Reach(\partial\Omega)$ see Definition \ref{positivereach} and thereafter.
\end{remark}

\begin{figure}
\centering
\includegraphics[width=3in]{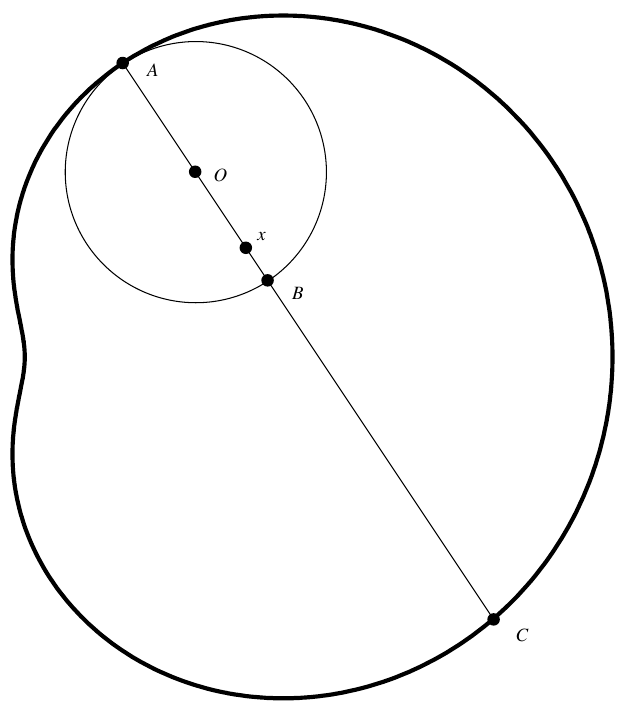}
\caption{Here $A:= \pr_{\complement\Omega} (x)$ and $l(x)=|A-C|\geq |A-B|=2r$ is the length of the shortest segment inside $\Omega$ containing $x$ and having direction $\frac{x- \pr_{\complement\Omega} (x)}{|x- \pr_{\complement\Omega} (x)|}.$}
\label{elle}
\end{figure}

The following consequence of \emph{Bernstein Inequality} will play a central role in our construction.
\begin{proposition}
 Let $\Omega$ be a bounded domain in $\R^d$ and let us introduce the sequence of functions
\begin{equation}
 \phi_n(x):= \begin{cases} \frac{n}{\sqrt{d_{\complement\Omega}(x)(l_\Omega-d_{\complement\Omega}(x))}}, & \mbox{if } d_{\complement\Omega}(x)<l_\Omega \\ \frac{n}{d_{\complement\Omega}(x)}\;\;\;, & \mbox{otherwise} \end{cases}.
\label{phi}
\end{equation}
For any $x\in \Omega$ let $v\in\{\frac{x-y}{|x-y|}:y\in \pr_{\complement\Omega} (x)\}$, then for any $p\in \mathscr P^n(\R^d)$ we have
\begin{equation}
 |\partial_v p(x)|\leq \phi_n(x)\|p\|_K.\label{derivativeestimate}
\end{equation}

If moreover we have $l_\Omega>0$, let us pick any $0<\delta<l_\Omega$ and define the sequence of functions
\begin{equation}
 \phi_{n,\delta}(x):= \begin{cases} \frac{n}{\sqrt{d_{\complement\Omega}(x)(\delta-d_{\complement\Omega}(x))}}, & \mbox{if } d_{\complement\Omega}(x)<\delta \\ \frac{n}{d_{\complement\Omega}(x)}\;\;\;, & \mbox{otherwise} \end{cases}.
\label{phidelta}
\end{equation}
Then the above polynomial estimate \eqref{derivativeestimate} still holds when $\phi_{n}$ is replaced by $\phi_{n,\delta}.$
\label{bernsteinsegment}
\end{proposition}
\begin{proof}
Pick $p\in \mathscr P^n(\R^d)$.
 Let us take $x\in\Omega$ such that $d_{\complement\Omega}(x)<l_\Omega.$ We denoted by $S_v(x)$ the segment $x+[-d_{\complement \Omega}(x),l_\Omega-d_{\complement \Omega}(x)]v,$ where $v$ is as above and $x\in S_v(x)$ due to $d_{\complement\Omega}(x)<l_\Omega.$ The restriction of $p$ to this segment is an univariate polynomial $q(\xi):=p(x+v\xi)$ of degree not exceeding $n$, then we can use the Bernstein Inequality \ref{bernsteinth} to get
\begin{equation*}
 \left|\frac{\partial q}{\partial\xi}(\xi)\right|\leq\frac{n}{\sqrt{(\xi+d_{\complement\Omega}(x))(l_\Omega-d_{\complement\Omega}(x)-\xi)}}    \|p\|_{S_v(x)},
\end{equation*}
evaluating at $\xi=0$ we get
\begin{equation}
 |\partial_v p(x)|\leq \frac{n\|p\|_{S_v(x)}}{\sqrt{d_{\complement\Omega}(x)(l_\Omega-d_{\complement\Omega}(x))}} \leq \frac{n\|p\|_K}{\sqrt{d_{\complement\Omega}(x)(l_\Omega-d_{\complement\Omega}(x))}} ,
\label{segmentestimate}
\end{equation}
 thus establishing the first case of \eqref{phidelta}.

Let $x$ be such that $d_{\complement\Omega}(x)\geq l_\Omega$. Notice that $B(x,d_{\complement\Omega}(x))\subseteq\Omega$ and hence $\forall \eta\in \mathbb S^{d-1}$ (the standard unit $d-1$ dimensional sphere) we can pick a segment in the  direction of $\eta$ having length $d_{\complement\Omega}(x)$ lying in $K$ and having $x$ as midpoint.
The Bernstein Inequality gives
\begin{equation}
 \label{ballestimate}
|\partial_v p(x)|\leq\max_{\eta\in \mathbb S^{d-1}}|\partial_\eta p(x)|\leq \frac{n}{d_{\complement\Omega}(x)}\|p\|_{B(x,d_{\complement\Omega}(x))}\leq \frac{n}{d_{\complement\Omega}(x)}\|p\|_K.
\end{equation}

The last statement follows directly by the special choice of $\delta<l_\Omega$. The right hand side    in \eqref{phidelta} dominates (case by case) the r.h.s.    in \eqref{phi} when cases are chosen accordingly to \eqref{phidelta}. 
\Square\end{proof}
 Actually the above proof proves also the following corollary, it suffices to take \eqref{phidelta} and substitute $\frac{n}{d_{\complement\Omega}(x)}$ by $\frac{n}{\delta}$ in the second case.

\begin{cor}
 Let $\Omega$ be an open bounded domain and $\delta$ a positive number such that $K_\delta:=\{x\in \Omega:d_{\complement \Omega}(x)\geq\delta\}\neq \emptyset$. Then for any $v\in \mathbb S^{d-1}$ we have $\forall p\in \mathscr P^n(\R^d)$
\begin{equation}
 \|\partial_v p\|_{K_\delta}\leq \frac{n}{\delta}\|p\|_K.
\label{Kdeltaestimate}
\end{equation}
\end{cor}

 We introduce the following in the spirit of \cite{V12}. Let us denote by $ds(\cdot)$ the standard length measure in $\R^d$.

\begin{proposition}
 Let $\Omega$ be a bounded domain in $\R^d$ such that $l_\Omega>0$ and let $0<\delta\leq l_\Omega$. Then

\begin{enumerate}[(i)]
\item for any $x\in \Omega$ the map 
\begin{equation*}
  \pr_{\complement\Omega} (x)\ni y\mapsto \int_{[y,x]}\phi_{n,\delta}(\xi) ds(\xi)
\end{equation*}
is constant, let $F_{n,\delta}(x)$ be its value.
\item We have
\begin{equation}
 F_{n,\delta}(x)=\begin{cases} n\arcos(1-\frac{2d_{\complement\Omega}(x)}{\delta}), & \mbox{if } d_{\complement\Omega}(x)<\delta \\ n\left(\pi+\ln{\frac{d_{\complement\Omega}(x)}{\delta}}\right)\;\;\;, & \mbox{otherwise}. \end{cases}
\label{integral}
\end{equation}
In particular $F_{n,\delta}$ extends continuously to $\overline\Omega$.
\item $F_{n,\delta}$ is constant on any level set of $d_{\complement\Omega}(\cdot)$ and $\sup_{\Omega\setminus K_\delta}F_{n,\delta}=n\pi$.

Let us set $a_{n,\delta}^i:=\frac{i n \pi}{ m_n}$ where $i=0,1,\dots m_n$ and $m_n$ is any positive integer greater than $2 n\pi$, we denote by $\Gamma_{n,\delta}^i$ the $a_{n,\delta}^i$-level set of $F_{n,\delta}$.
\item  We have
\begin{eqnarray*}
 \Gamma_{n,\delta}^i&=&\{x\in K:d_{\complement \Omega}(x)=d_{n,\delta}^i\}\;\;\text{, where}\\
d_{n,\delta}^i&:=&\frac{\delta}{2}\left(1-\cos\left(\frac{i\pi}{m_n}\right)\right).
\end{eqnarray*}
\item Let $\Gamma_{n,\delta}:=\cup_{i=0}^{m_n}\Gamma_{n,\delta}^i,$ then for any $p\in \mathscr P^n(\R^d)$ we have
\begin{equation}
 \|p\|_K\leq \max\{ 2\|p\|_{\Gamma_{n,\delta}}, \|p\|_{K_\delta} \}.\label{picewiseestimate}
\end{equation}
\end{enumerate}
\label{preimage}
\end{proposition}

\begin{figure}[ht]
\centering
\includegraphics[width=4in]{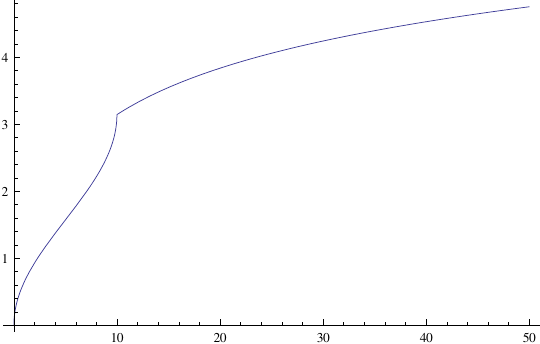}
\caption{A plot of a section of $F_{n,\delta}$ along a segment of metric projection, where $\delta=10$, $n=1$. Abscissa here is the distance from the boundary. }
\label{Fn}
\end{figure}

\begin{proof}

\vskip 1em\noindent(\emph{i}) The function $\phi_{n,\delta}(\cdot)$ depends on its argument only by the distance function, $\phi_{n,\delta}(x)=:g_{n,\delta}(d_{\complement\Omega}(x))$. The length of the segment $[y,x]$ is clearly constant when $y$ varies in the set $\pr_{\complement \Omega}(x)$. 

Moreover for any $y,z\in\pr_{\complement \Omega}(x)$ let us denote by $R_{y,z}$ an euclidean isometry that maps $[y,x]$ onto $[z,x]$, one trivially has $d_{\complement\Omega}(\xi)=d_{\complement\Omega}(R_{y,z}\xi)$ for any $\xi\in[y,x]$. This is because $ \pr_{\complement\Omega} (\xi)\ni y$ for any $\xi\in[x,y]$ by the Triangle Inequality and thus $d_{\complement\Omega}(\xi)=|\xi-y|$.

Thus we have
\begin{eqnarray*}
& & \int_{[y,x]}\phi_{n,\delta}(\xi) ds(\xi)=\int_{[y,x]}g_{n,\delta}(d_{\complement\Omega}(\xi))ds(\xi)\\&=&\int_{[y,x]}g_{n,\delta}(d_{\complement\Omega}(R_{y,z}\xi))ds(\xi)=\int_0^1 g_{n,\delta}\left(d_{\complement\Omega}\left(R_{y,z}\left(y+t\frac{x-y}{|x-y|}\right)\right)\right)dt\\
&=&\int_0^1 g_{n,\delta}\left(d_{\complement\Omega}\left(z+t\frac{z-x}{|z-x|}\right)\right)dt=\int_{[z,x]}\phi_{n,\delta}(\eta) ds(\eta).
\end{eqnarray*}

\vskip 1em\noindent(\emph{ii}) Let us parametrize the segment as $y+s\frac{x-y}{|x-y|}$, then we have
\begin{equation}
 F_{n,\delta}(x)=\begin{cases} \int_0^{d_{\complement\Omega}(x)}\frac{n}{\sqrt{s(\delta-s)}}d s, & \mbox{if } d_{\complement\Omega}(x)<\delta \\ 
\int_0^{\delta}\frac{n}{\sqrt{s(\delta-s)}}d s\,+\int_\delta^{d_{\complement\Omega}(x)}\frac{n}{s}d s\;\;\;, & \mbox{otherwise}. \end{cases}
\end{equation}
The first integral can be solved by substitution: $s=\frac\delta 2 (1-\cos\theta)$. The integration domain becomes $[0,\theta_x]$ where $\frac \delta 2(1-\cos(\theta_x))=d_{\complement \Omega}(x)$, while the integral itself becomes $\int_0^{\theta_x}d\theta=\theta_x$, thus the first case in \eqref{integral} is proven. 

The second integral has an immediate primitive.
$F_{n,\delta}$ depends on $x$ only by the distance function, moreover we notice that
$$\lim_{s\rightarrow \delta^-}\arcos{\left(1-\frac{2 s}{\delta}\right)}=\pi=\lim_{s\rightarrow \delta^+}\left(\pi+\ln{\frac s \delta}\right),$$
hence $F_{n,\delta}$ is a continuous function of the distance function. Since $d_{\complement\Omega}$ is well known to be $1-$Lipschitz $F_{n,\delta}$ is continuous on $\Omega$.

Since $d_{\complement\Omega}$ extends continuously to $\overline\Omega$, then $F_{n,\delta}$ does. Actually we must take $F_{n,\delta}|_{\partial \Omega}\equiv 0.$

\vskip 1em\noindent(\emph{iii}) We already used that $F_{n,\delta}$ depends on $x$ only by the distance function and hence $F_{n,\delta}|_{d_{\complement\Omega}^{\leftarrow}(a)}=\text{constant}$\footnote{We denote by $f^{\leftarrow}(a)$ the inverse image under $f:D\rightarrow \R$ of the number $a\in \Range[f]$, i.e., $\{x\in D: f(x)=a\}$  that, in general, is a set.}, moreover the functions $\arcos{\left(1-\frac{2 s}{\delta}\right)}$ and $\left(\pi+\ln{\frac s \delta}\right)$ are both increasing in $[0,\max_{x\in\overline \Omega}d_{\complement\Omega}(x)],$  see Figure \ref{Fn}, hence any level set of $F_{n,\delta}$ must coincide with a suitable level set of the distance function.

\vskip 1em\noindent(\emph{iv}) The conclusion follows immediately by inverting the equation $$n\arcos\left(1-\frac{2 d_{n,\delta}^i}{\delta}\right)=a_{n,\delta}^i.$$

\vskip 1em\noindent(\emph{v}) Let $p\in \mathscr P^n(\R^d)$ be fixed, let us pick $x\in K$, then two possibilities can occur. In the first case $x\in K_\delta.$ In this case we have $|p(x)|\leq \|p\|_{K_\delta}.$ In the second we suppose $x\notin K_{\delta}$, let us consider $y\in  \pr_{\complement\Omega} (x)$. The segment $[y,x]$ cuts $\Gamma_{n,\delta}^i$ for every $i$ such that $ d_{n,\delta}^i\leq d_{\complement\Omega}(x)$, moreover $[y,x]\cap\Gamma_{n,\delta}^i=\{y^i\}$, due to the monotonicity of $F_{n,\delta}$ along any segment where $d_{\complement\Omega}$ is monotone.

Let $i(x):=\max\{i:d_{n,\delta}^i\leq d_{\complement\Omega}(x)\}$ and let $y^{i(x)+1}$ be the unique intersection of $\Gamma_{n,\delta}^{i(x)+1}$ and the ray starting from $x$ and having direction $\frac{x-y}{|x-y|}$. 

Let $s(\cdot)$ be the arc length parametrization of the segment $[y^{i(x)},y^{i(x)+1}]$ now we have
\begin{eqnarray*}
 |p(x)|&\leq&|p(y^{i(x)})|+ \int_0^{s^{-1}(x)} \left|\frac{ \partial ( p\circ s ) }{\partial t} (t)\right|dt\\
&\leq & |p(y^{i(x)})| + \int_0^1 \left|\frac{ \partial ( p\circ s ) }{\partial t} (t)\right|dt \\
&=& |p(y^{i(x)})| + \int_0^1\|p\|_K \phi_{n,\delta}(s(t)) dt\\&=&|p(y^{i(x)})| + \int_{[y^{i(x)},y^{i(x)+1}]} \|p\|_K\phi_{n,\delta}(\xi) ds(\xi)\\
&\leq&|p(y^{i(x)})| +\frac{\|p\|_K}{m_n} \int_{[y^0,y^{m_n}]}\phi_{n,\delta}(\xi) ds(\xi)\\
&\leq& \|p\|_{\Gamma_{n,\delta}^{i(x)}}+\frac{F_{n,\delta}(y^{m_n})}{m_n}\|p\|_K\leq\|p\|_{\Gamma_{n,\delta}^{i(x)}}+\frac 1 2 \|p\|_K,
\end{eqnarray*}
where we used \eqref{derivativeestimate} in the third line while the special choice of $a_{n,\delta}^i$ (and thus $y^i$) as equally spaced points in the image of $F_{n,\delta}$ and the choice of $m_n>2 n \pi$ has been used in the last two lines.

To conclude we take the maximum of the above estimates among $x\in K$ thus letting $i$ varying among $0,1,\dots,m_n-1$ and considering both cases $x\in K_\delta$ and $x\notin K_\delta.$ 
\Square\end{proof}

\begin{proposition}\label{manifold}
 Let $\Omega$ be a bounded $\mathscr C^{1,1}$ domain, $0<r<\Reach(\partial\Omega)$ $0<\delta\leq r$ and let $m_n>2n\pi$, then
\begin{enumerate}[(i)]
 \item For any $i=1,\dots m_n$ $\Gamma_{n,\delta}^i$ is a $\mathscr C^{1,1}$ hypersurface.
\item For any $p\in \mathscr P^n(\R^d)$ any $x\in \Gamma_{n,\delta}^i$ and any $v\in \mathbb S^{d-1}\cap \mathcal T_x\Gamma_{n,\delta}^i\;\text{ where }i=0,1,\dots,m_n$ we have
\begin{equation}
|\partial_v p(x)| \leq 
\begin{cases}
\frac{n}{\delta}\|p\|_K                                          & i=0\\
\frac{2n}{\delta}\|p\|_K  & i=1,2,\dots,m_n
                       \end{cases}.
\label{buccia}
\end{equation}
\end{enumerate}
\label{tangestimatepita}
\end{proposition}

\begin{figure}[ht]
\centering
\begin{tabular}{cc}
\includegraphics[width=2.5in]{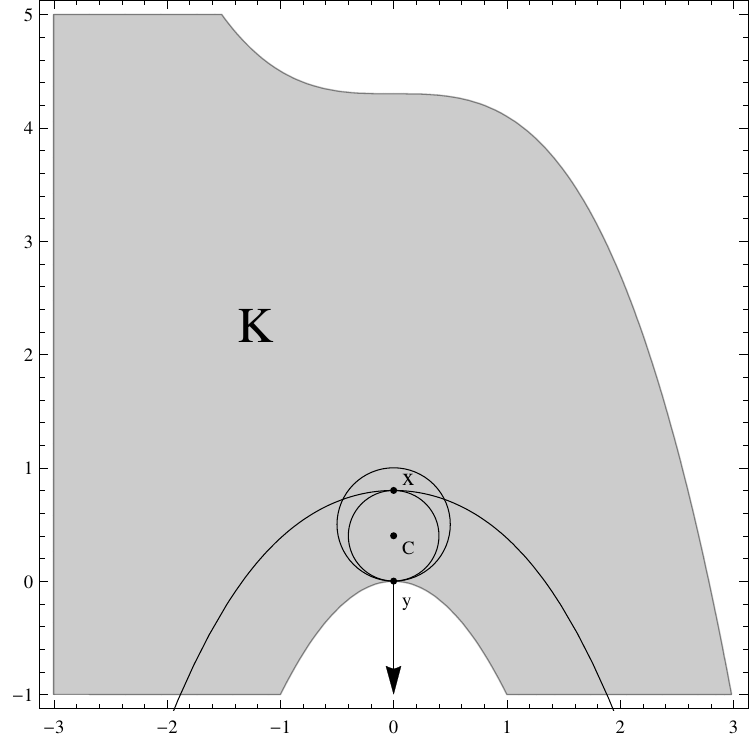} & \includegraphics[width=2.5in]{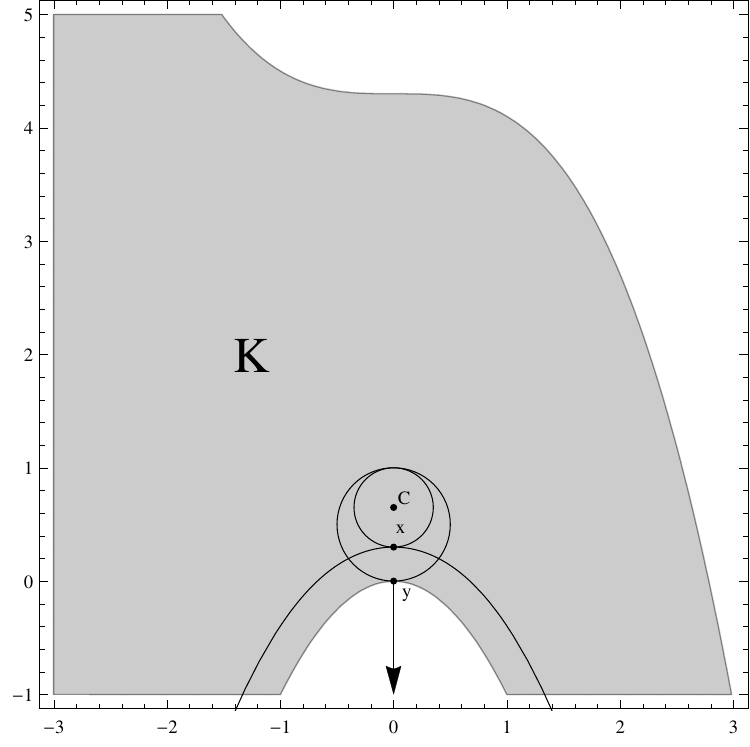}
\end{tabular}
\caption{Different situations occurring in the proof of Proposition \ref{manifold} (\emph{ii}). On the left side the tangent ball at $x$ is chosen outward and on the left side inward, this corresponds respectively to the first and the second case in \eqref{palle}. The arrow represent $\nabla b_{\Omega}(x).$}
\label{balls}
\end{figure}

\begin{proof}
\begin{enumerate}[(i)]
 \item Notice that we have, due to \ref{regularirtyoforienteddistance}, 
$$0 < \min\{\Reach(\Omega),\Reach(\complement\Omega)\}=\Reach(\partial\Omega).$$

If $i>0$ due to \eqref{orienteddistancegradient} and Theorem \ref{regularirtyoforienteddistance}. We have $\forall x\in\Gamma_{n,\delta}^i$
$$\nabla d_{\complement\Omega}(x)=-\nabla b_{\Omega}(x)=\frac{x-\pr_{\partial\Omega}(x)}{|x-\pr_{\partial\Omega}(x)|},$$
moreover this is a Lipschitz function when restricted to $\{|b_\Omega(x)|<\delta\}$ for any $0<\delta<\min\{\Reach(\Omega),\Reach(\complement\Omega)\}.$

Also we have $b_\Omega |_{\overline\Omega}\equiv-d_{\complement\Omega}.$ 

We notice that $\nabla d_{\complement \Omega}(x) \neq 0$, therefore any level-set of $d_{\complement\Omega}$ contained in $\Omega\setminus K_\delta$ is a $\mathscr C^{1,1}$ $d-1$ dimensional manifold by the Implicit Function Theorem.
\item If $i=0$ Theorem \ref{regularirtyoforienteddistance} tells that for any $x$ in $\Gamma_{n,\delta}^i$ we have $B_x:=B(x+\delta\nabla b_{\Omega}(x),\delta)\subseteq\Omega,$ (cfr. figure \ref{balls} point $C1$) moreover $\mathcal T_x\Gamma_{n,\delta}^i=\mathcal T_x\partial B_x.$ Therefore we can apply the Markov Tangential Inequality to the ball $B_x:$ for any polynomial $p\in \wp^n$ and any $u\in\mathcal T_x\Gamma_{n,\delta}^i =\mathcal T_x\partial B_x$ we have
\begin{equation}
|\partial_u p(x)|\leq \frac n \delta \|p\|_{B_x}\leq\frac n \delta \|p\|_{K}.
\end{equation} 
Where the last inequality follows from $\overline B_x\subseteq K.$
 
Now we focus on $i>0.$ Let us take $x\in \Gamma_{n,\delta}^i,$ then $y= \pr_{\complement\Omega} (x)$ $\Rightarrow$ $\nabla b_\Omega(y)=\nabla b_\Omega(x)$ and hence we have $\mathcal T_x \Gamma_{n,\delta}^i=\mathcal T_{y} X, i=0,1,\dots,m_n$

Moreover we notice that \small
\begin{eqnarray}
B_x^i&:=&
\begin{cases}
B\left(y+\frac {d_{n,\delta}^i} 2 \nabla b_{\Omega}(x),\frac {d_{n,\delta}^i} 2\right)\subset\Omega& d_{n,\delta}^i\geq \delta/2\\
B\left(y+(d_{n,\delta}^i+\frac {2\delta-d_{n,\delta}^i} 2) \nabla b_{\Omega}(x),\frac {2\delta-d_{n,\delta}^i} 2\right)\subset\Omega& d_{n,\delta}^i< \delta/2.
\end{cases}\label{palle}\\
\mathcal T_x \Gamma_{n,\delta}^i&=&\mathcal T_x B_x^i.
\end{eqnarray}
This can be figured out by looking at in Figure \ref{balls} where the first occurrence is represented on the left and the second on the right.

Now we notice that the radius of $B_x^i$ can be bounded below uniformly in $i$ by $\delta/2.$ Therefore The Markov Tangential Inequality for the ball gives us the following: $\forall p\in \wp^n$ and $\forall v\in\mathcal T_x \Gamma_{n,\delta}^i,$ $|v|=1$ we have 
\begin{equation*}
 |\partial_v p(x)|\leq \frac{n}{\delta/2}\|p\|_{B_x^i}.
\end{equation*}

Now due to $\mathcal T_x \Gamma_{n,\delta}^i=\mathcal T_x B_x^i$ and $B_x^i\subset \Omega$ we have $\forall p\in \wp^n,v\in \mathcal T_x \Gamma_{n,\delta}^i,|v|=1,\forall i=0,1,m_n $
\begin{equation*}
 |\partial_v p(x)|\leq \frac{n}{\delta/2}\|p\|_{K}.
\end{equation*}
\end{enumerate}
\Square\end{proof}

\subsection{Proof of the main result}
We developed all required tools to state and prove the main result of this paper, Theorem \ref{mainresult2}. The idea of its constructive proof is mixing the technique of Theorem \ref{mainresult} with an improvement of the one being used in \cite{CL08}[Th. 5]. More precisely the hypersurfaces $Z_n$ of Theorem \ref{mainresult} here are replaced  by the level sets $\Gamma_{n,\delta}^i$ which together with the set $K_\delta = \{x\in K:d_{\complement\Omega}(x)\geq \delta\}$ are shown to form a norming set for $K.$ 
\begin{theorem}\label{mainresult2}
 Let $\Omega$ be a bounded $\mathscr C^{1,1}$ domain in $\R^d$, then there exists an optimal admissible mesh for $K:=\overline \Omega$. 
\label{2mainresult}
\end{theorem}

\begin{proof}
Notice that we have $0 < \min\{\Reach(\Omega),\Reach(\complement\Omega)\}=\Reach{\partial\Omega}$ due to \ref{regularirtyoforienteddistance} we fix $\delta\leq r<\Reach{\partial\Omega}$ 

Let us recall the above notation
\begin{eqnarray*}
 K_\delta &:=& \{x\in K:d_{\complement\Omega}(x)\geq \delta\},\\
\Gamma_{n,\delta} &:=&  \cup_i \Gamma_{n,\delta}^i\text{ where}\\
\Gamma_{n,\delta}^i &:=& \{x\in K:d_{\complement\Omega}(x)= d_{n,\delta}^i\}, \\
d_{n,\delta}^i&:=&\frac{\delta}{2}\left(1-\cos\left(\frac{i\pi}{m_n}\right)\right)\text{ , where we can take}\\
m_n &:=&\lceil 2 n \pi \rceil +1.
\end{eqnarray*}
 
Let $p\in \mathscr P^n(\R^d)$.

\noindent$\bullet$ \textbf{Claim 1.} \emph{For any $\lambda>1$ there exists $Z_{n,\delta,\lambda}\subset K_\delta$ such that
\begin{eqnarray}\label{claim1}
 \|p\|_{K_\delta}&\leq&\|p\|_{Z_{n,\delta,\lambda}}+\frac{1}{\lambda}\|p\|_K\;\text{and}\\
\Card{Z_{n,\delta,\lambda}}&=&O(n^d).
\end{eqnarray}
}
\noindent$\bullet$ \textbf{Proof of Claim 1.} Let us consider for any $\lambda>1$ a mesh $Z_{n,\delta,\lambda} $ such that its fill distance
\begin{equation*}
h(Z_{n,\delta,\lambda})\leq\frac{\delta}{\lambda n+1/2}=:h\text{ , see \eqref{clfilldistance}.}
\end{equation*}

Let us define $Z_{n,\delta,\lambda}\subset K_\delta$ as the intersection of $K$ with a grid $G$ with a step-size $\frac{h}{\sqrt d}$ on a suitable $d$ dimensional cube containing $K.$ It follows that $\Card(Z_{n,\delta,\lambda})=\left(\frac{\sqrt d} h\right)^d=O(n^d).$

Now pick any $x\in K_\delta$ and find $y\in Z_{n,\delta,\lambda}$ such that $|x-y|\leq h$ and define $v:=\frac{x-y}{|x-y|}$ and notice that
\begin{eqnarray*}
& & |p(x)|\\
&\leq& |p(y)|+\left|\int_{0} ^{|x-y|}\partial_v p(x+sv)ds\right|\leq \|p\|_{Z_{n,\delta,\lambda}}+|x-y|\|p\|_{[x,y]}\\
  &\leq&\|p\|_{Z_{n,\delta,\lambda}}+\|\partial_v p\|_{B(K_\delta,h/2)}.
\end{eqnarray*}
Where we used $\min_{\xi\in[x,y]}\dist(\xi,K_\delta)\geq h/2$ due to the Triangle Inequality for the euclidean distance $\dist(\cdot,K_\delta)$ from $K_\delta.$

By the observation $B(K_\delta,h/2)\subseteq K_{\delta-h/2}$ we can apply inequality  \eqref{Kdeltaestimate} where $\delta$ is replaced by $\delta-h/2.$

$$|p(x)|\leq |p(y)|+h \frac{n}{\delta-h/2}\|p\|_K$$

Taking maximum over $x \in K_\delta$ and using the particular choice $h:=\frac{\delta}{\lambda n+1/2}$ we are done.

\noindent$\bullet$ \textbf{Claim 2.} \emph{ For any $2<\mu$ there exist finite sets $Y_{n,\delta}^i\subset\Gamma_{n,\delta}^i$, $i=0,1,..m_n$, such that if we set $Y_{n,\delta}:=\cup_i Y_{n,\delta}^i$ we get}
\begin{eqnarray}\label{claim2}
 \|p\|_{\cup_i\Gamma_{n,\delta}^i}&\leq&\|p\|_{Y_{n,\delta}}+\frac{1}{\mu}\|p\|_K\;\text{and}\\
\Card{Y_{n,\delta}}&=&O(n^d).
\end{eqnarray}

\noindent$\bullet$\textbf{Proof of Claim 2.} Let us pick $Y_{n,\delta}^i\subset\Gamma_{n,\delta}^i$ such that
\begin{equation}
 h_{\Gamma_{n,\delta}^i}(Y_{n,\delta}^i)\leq
\begin{cases}                                 
\frac{\delta}{\mu n}                                          & i=0\\
\frac{\delta}{2\mu n}  & i=1,2,\dots,m_n
                                            \end{cases} \text{ (see Definition \ref{geofill}).}
\label{filldistanceassumption}
\end{equation}

 Now fix any $i\in \{0,1,\dots,m_n\}$, by \eqref{filldistanceassumption} for any $x\in\Gamma_{n,\delta}^i$ there exist a point $y\in Y_{n,\delta}^i$ and a Lipschitz curve\footnote{Notice that $\Gamma_{n,\delta}^i$ are compact $\mathscr C^{1,1}$ hypersurfaces, thus in particular they are locally complete with respect the geodesic distance. Therefore there exists a curve $\gamma$ realizing the infimum in the definition of geodesic fill distance.} $\gamma$ lying in $\Gamma_{n,\delta}^i$, connecting $x$ to $y$ and such that $\Var[\gamma]\leq h_{\Gamma_{n,\delta}^i}(Y_{n,\delta})$ . Let us denote the arclength reparametrization of $\gamma$ by $\tilde \gamma$, then we have
\begin{eqnarray*}
|p(x)|&\leq& |p(y)|+\int_0^{\Var[\gamma]} \ \frac{d(p\circ\tilde\gamma)}{dt}(t) dt\\
&\leq&  \|p\|_{Y_{n,\delta}^i}+h_{\Gamma_{n,\delta}^i}(Y_{n,\delta})\max_{\xi\in \Gamma_{n,\delta},v\in \mathbb S^{d-1}\cap\mathcal T_\xi \Gamma_{n,\delta}^i}|\partial_v p(\xi)|\\
&\leq& \|p\|_{Y_{n,\delta}^i}+\frac 1 \mu \|p\|_K.
\end{eqnarray*}
Here, in the 3rd line, we used the inequality \eqref{buccia}.
Let us take the maximum w.r.t. $x$ varying in $\Gamma_{n,\delta}^i$ and $i$ varying over $\{0,1,\dots,m_n\}$, we obtain $
 \|p\|_{\Gamma_{n,\delta}}\leq \|p\|_{Y_{n,\delta}}+\frac{1}{\mu}\|p\|_K.$

We are left to prove that we can pick $Y_{n,\delta}^i$ such that $\Card(Y_{n,\delta})=O(n^d).$

When $i=0$ Proposition \ref{boundarymesh} ensures ($X$ is a $\mathscr C^{1,1}$ hypersurface and a fortiori is Lipshitz) the existence of such an $Y_{n,\delta}^0$ with $h_{\Gamma_{n,\delta}^0}(Y_{n,\delta}^0)\leq \frac{\delta}{\mu n}$ and $\Card (Y_{n,\delta}^0)=O(n^{d-1})$. Let us study the case $i>0$.

Now let us notice that  by \emph{(v)} in Theorem \ref{regularirtyoforienteddistance} one has\\ $\pr_{\partial\Omega}|_{b_\Omega=\rho}$ is an injective function for any $0<\rho<\Reach(\partial\Omega).$ Since $\nabla b_\Omega$ constant along metric projections we can also notice that $\nabla b_\Omega(x)=\nabla b_\Omega(\pr_{\partial\Omega}(x)).$ Moreover by \emph{(iii)} in Theorem \ref{regularirtyoforienteddistance} if $x\in \Gamma_{n,\delta}^i,$ $y=\pr_{\complement\Omega(x)}$ then 
\begin{eqnarray*}
& & y=\pr_{\complement\Omega}(x)=x-|x- \pr_{\complement\Omega} (x)|\nabla b_\Omega(x)\\
&=&x-d_{n,\delta}^i\nabla b_\Omega(x)=x-d_{n,\delta}^i\nabla b_\Omega(\pr_{\partial\Omega}(y))\\
&=&x-d_{n,\delta}^i\nabla b_\Omega(y).
\end{eqnarray*}

Thus we can introduce the family of inverse maps $f_i:=\left( \pr_{\complement\Omega} |_{\Gamma_{n,\delta}^i}\right)^{-1}$
\begin{eqnarray*}
 f_{i}:\Gamma_{n,\delta}^0&\longrightarrow&\Gamma_{n,\delta}^i\\
x&\longmapsto&x+ d_{n,\delta}^i \nabla b_\Omega(x).
\end{eqnarray*}

Notice that $\nabla b_\Omega|_{\partial\Omega}$ is a Lipschitz function, see Theorem \ref{regularirtyoforienteddistance} \emph{(iii)}. Let us denote $L$ its Lipschitz constant.

Therefore  $\{f_i\}_{i=1,2,\dots,m_n}$ is a family of equi-continuous functions of Lipschitz constant $$\max_{i=1,2,\dots,m_n}(1+Ld_{n,\delta}^i)\leq(1+L\delta).$$

Now the \emph{Area Formula} says that $f_i$ (being $1+ L\delta$ Lipschitz) maps a mesh of $\Gamma_{n,\delta}^0$ with geodesic fill distance $\frac{h}{1+\delta L}$ onto a mesh in $\Gamma_{n,\delta}^i$ having geodesic fill distance bounded by $h$. We already used this property and explained its application in more detail in the proof of Theorem \ref{mainresult}, see \eqref{area} and thereafter.

Thanks to Proposition \ref{boundarymesh} we can pick the mesh $\tilde Y_{n,\delta}^i\subset\Gamma_{n,\delta}^0$ such that $h_{\Gamma_{n,\delta}^0}(\tilde Y_{n,\delta}^i)\leq \frac{\delta}{2\mu n(1+\delta L)}$ with the cardinality bound  $\Card(\tilde Y_{n,\delta}^i)=O(\left(\frac{n}{h}\right)^{d-1})$ where we denote $\frac{\delta}{2\mu(1+\delta L)}$ by $h$. Let us set $Y_{n,\delta}^i:=\{f_i(y),y\in\tilde Y_{n,\delta}^i\}$. Now we can notice that
\begin{equation*}
 \Card (Y_{n,\delta})=\sum_{i=0}^{m_n}\Card Y_{n,\delta}^i=n^{d-1}+\sum_{i=1}^{m_n}O\left(\left(\frac{n}{h}\right)^{d-1}\right)=O(n^d).
\end{equation*}

\noindent$\bullet$ \textbf{Claim 3:} \emph{$A_{n,\delta}:=Y_{n,\delta}\cup Z_{n,\delta,\lambda}$ is an optimal admissible mesh for $K.$}

\noindent$\bullet$ \textbf{Proof of Claim 3.} By the special choice of $\delta<r\leq l_\Omega/2$ we can use jointly \eqref{picewiseestimate}, \eqref{claim1} and \eqref{claim2} and we obtain
\begin{equation*}
 \|p\|_K\leq\max\{ 2\|p\|_{Y_{n,\delta}}+2\frac{1}{\mu}\|p\|_K\;,\; \|p\|_{Z_{n,\delta,\lambda}}+\frac{1}{\lambda}\|p\|_K\}.
\end{equation*}
By the elementary properties of $\max$ we have
\begin{equation}
 \|p\|_K\leq\max\{\frac{2\mu}{\mu-2}\,,\, \frac{1}{\lambda-1}\}\|p\|_{Y_{n,\delta}\cup Z_{n,\delta,\lambda}}.
\end{equation}
Thus  $Y_{n,\delta}\cup Z_{n,\delta,\lambda}=:A_{n,\delta}$ satisfies
\begin{equation}
 \|p\|_K\leq C(\delta,\lambda,\mu)\|p\|_{A_{n,\delta}}\;\forall p\in \mathscr P^n(\R^d)\;\forall n\in \N
\end{equation}
has the correct cardinality order of growth.

\Square\end{proof}

\section{Acknowledgements}
The author would like to thank prof. M. Vianello (Universit\'a degli Studi di Padova) for his constant support, prof. R. Monti (Universit\'a degli Studi di Padova) for many interesting discussions and explanation, prof. Len Bos (Universit\'a di Verona) for his helpfulness and easiness  and prof. A. Kro\'o (Alfr\'ed R\'enyi Institute of Mathematics Hungarian Academy of Sciences Budapest, Hungary) for hints on the specific topic and for sharing his unpublished material, Khay Nguyen and D. Vittone for some interesting suggestions. Last but not the least the author would like to thank the anonymous referees: their extremely valuable comments improved in a relevant way this manuscript.

\appendix
\section{Sets of positive reach}
Here we provide very concisely some essential tools that we use in the proofs of the paper. Of course we do not even try to be exhaustive, since this is far from our aim.

We deal with Federer sets of positive reach, they were introduced in the outstanding article \cite{Fed59}.
\begin{definition}[Reach of a Set]\cite{Fed59}\label{positivereach}
Let $A\subset\R^d$ be any set, we denote by $\pr_A(x)=\{y\in A: |y-x|=d_A(x)\}$ the metric projection onto $A$, where we denoted by $d_A(x):=\inf_{y\in A}|x-y|.$ Moreover let $Unp(A):=\{x\in \R^d: \exists ! y\in A,\; \pr_A(x)=\{y\}\}.$ Then we define 
\begin{eqnarray}
 \Reach(A,a)&:=&\sup_{r>0}\{r: B(a,r)\subseteq Unp(A) \} \text{ for any }a\in A,\\
\Reach(A)&:=&\inf_{a\in A}\Reach(A,a).
\end{eqnarray}
The set $A$ is said to be a set of positive reach if $\Reach(A)>0.$  
\end{definition}

By this definition sets of reach $r>0$ are precisely the subsets of $\R^d$ for which there exists a tubular neighborhood of radius $r$ where the metric projection is unique and moreover this tubular neighborhood is maximal.

This class of sets was introduced by Federer in the study of \emph{Steiner Polynomial} relative to a (very smooth) set, the polynomial that computed at $r>0$ gives the $d$-dimensional  measure of the $r$ tubular neighborhood of the given set. The main interest on such a class of sets is that under this assumption (in place of high degree of smoothness) one can recover the coefficients of Steiner Polynomial as Radon measures, the Curvature Measures.

Sets with positive reach may be seen as a generalization of $\mathscr C^{1,1}$ bounded domains, in fact the latter can be characterized as domains such that the boundary has positive reach, a more restrictive condition. Moreover if $\Omega$ is a domain having positive reach it can be shown that the subset of $\partial \Omega$ where the distance function defines uniquely a normal vector field (as for $\mathscr C^{1,1}$ domains) is ``big'' in the right measure theoretic sense.

However, from our point of view the most relevant feature of sets of positive reach is the one concerning the regularity properties of the distance function $d_A(\cdot).$ They can be found in \cite{Fed59}[Section 4]. If $A$ has positive reach then $d_A(\cdot)$ is differentiable at any point of $\R^d\setminus A$ having unique projection and we have $\nabla d_A(x)=\frac{x-\pr_A(x)}{d_A(x)}$ and this is a Lipschitz function in any set of the type $\{x:0<s\leq d_A(x)\leq r<\Reach(A)\}.$

In the sequel of the paper we need to use a little of tangential calculus on non-smooth structures, so we introduce the following.

\begin{definition}[Tangent and Normal]
 Let $A\subset\R^d$ be\\ any set. Let $a\in A$ then we define respectively the tangent and the normal set to $A$ at the point $a$ as
\small
\begin{eqnarray*}
 \Tan(A,a)&:=&\left\{ u\in \R^d: \forall\epsilon>0\,\exists x\in A:\; |x-a|<\epsilon,\,\left|\frac u {|u|}-\frac{x-a}{|x-a|}\right|<\epsilon\right\}\\
\Nor(A,a)&:=&\left\{ v\in \R^d:\langle v,u\rangle\leq 0\;\forall u\in \Tan(A,a) \right\}.
\end{eqnarray*}
\normalsize
\label{tangentnormal}\end{definition}
Here the idea is to take all possible sequences $x_n\in A$ approaching $a$ and take the limit of $\frac{x_n-a}{|x_n-a|}$. For the normal set in the above definition the $\leq$ is preferred to the equality sign to allow to consider the non-smooth case and to work with more flexibility. The set $\Nor(A,a)$ actually is in general a cone given by the intersection of all half spaces \emph{dual}\footnote{Hereafter the word dual must be intended in the following sense \cite{Fed59}, $u$ is dual to $N\subset\R^d$ iff $\langle u,v\rangle\leq 0$ for any $v\in N.$} to a vector of $\Tan(A,a).$   

The notion of normal vector we introduced should be compared with other possible notions, the most relevant one is that of \emph{proximal calculus.}
\begin{definition}[Proximal Normal]
 Let $A\subset\R^d$ and $x\in \partial A.$ The vector $v\in \mathbb S^{d-1}$ is said to be a proximal normal to $A$ at $x$ (and we write $v\in N_A^P(x)$) iff there exists $r>0$ such that
\begin{equation}
 \left\langle v,\frac{y-x}{|y-x|}\right\rangle\leq\frac{1}{2r}|y-x|,\;\forall y\in \partial A.
\label{balleq}
\end{equation}
\end{definition}
Notice that the inequality \ref{balleq} implies that the boundary of $A$ lies outside of $B(x+r\frac{v}{|v|},r).$  If we focus on the boundary of a closed set the property of having non empty proximal normal set to the complement at each point of the boundary, i.e. $$N_{\complement\Omega}^P(x)\neq\emptyset\;\forall x\in\partial\Omega$$ is known as \textbf{Uniform Interior Ball Condition (UIBC)} and it is usually stated in the following (equivalent) way
\begin{definition}\label{uniforminteriorball}
 Let $\Omega\subset\R^d$ be a domain, suppose that for any $x\in \partial\Omega$ there exists $y\in\Omega$ such that $B(y,r)\cap\complement\Omega=\emptyset$ and $x\in \partial B(y,r).$ Then $\Omega$ is said to admit the uniform Interior Ball Condition.
\end{definition} 
 Such a condition (and some variants) appears in the literature also as External Sphere Condition (w.r.t. the complement of the set) in the context of the study of some properties of Minimum Time function in Optimal Control \cite{MKV12}, while the previous nomenclature is more frequently used in the framework of regularity theory of PDE.

It is worthwile recalling that positive reach is a strictly stronger condition when compared to UIBC. Actually if a set $A$ has positive reach, then it satisfies the UIBC at each point $a$ of its boundary and \emph{in any direction} of $\Nor(A,a).$

We will use several times the following easy fact.
\begin{proposition}\label{reachball}
 Let $A\subset \R^d,$ $\gamma:[0,1]\rightarrow \partial A$ a Lipschitz curve, $r>0$ and let us suppose $\Reach(A)>r.$  Then we have
for a.e. $s\in]0,1[$ there exists $v\in \mathbb S^{d-1}$ such that
\begin{enumerate}[(i)]
\item $B_s:=B(\gamma(s)+rv,r)\subseteq A^c,$
\item $\gamma'(s)\in\mathcal T_{\gamma(s)}B_s.$
\end{enumerate}
\end{proposition}
\begin{proof}
Let us consider the arclength re-parametrization $\tilde\gamma$ of $\gamma$ that is a $1-$Lipschitz curve from $[0,\Var[\gamma]]$ to $\support\gamma.$ 
Notice that $\tilde\gamma$, being Lipschitz, is a.e. differentiable in $]0,\Var[\gamma][,$ Let $\Sigma_{\tilde \gamma}$ be the set of singular points of $\tilde\gamma$ and let moreover $t_0$ be a point in $]0,\Var[\gamma][\setminus\Sigma_{\tilde\gamma}.$

First we claim that $\tilde\gamma'(t_0)\in \Tan(A,\tilde\gamma(t_0)).$

By differentiability of $\tilde\gamma$ at $t_0$ we have
\begin{equation}
 \lim_{\begin{array}{cc}t\rightarrow t_0 \\ t\in [0,\Var[\gamma]]\setminus \Sigma_{\tilde\gamma}\end{array}}\frac{\tilde\gamma(t)-\tilde\gamma(t_0)}{t-t_0}=\tilde\gamma'(t_0).
\end{equation}
Thus, recalling that $|\tilde\gamma'(t)|=1\neq 0$ in a neighborhood of $t_0$, we have
\begin{equation*}
 \lim_{\begin{array}{cc}t\rightarrow t_0 \\ t\in [0,\Var[\gamma]]\setminus \Sigma_{\tilde\gamma}\end{array}}\frac{\tilde\gamma(t)-\tilde\gamma(t_0)}{t-t_0}\frac{|t-t_0|}{|\tilde\gamma(t)-\tilde\gamma(t_0)|}=\frac{\tilde\gamma'(t_0)}{|\tilde\gamma'(t_0)|}.
\end{equation*}
Therefore we have
\begin{equation*}
 \lim_{\begin{array}{cc}t\rightarrow t_0 \\ t\in [t_0,\Var[\gamma]]\setminus \Sigma_{\tilde\gamma}\end{array}}\left|\frac{\tilde\gamma'(t_0)}{|\tilde\gamma'(t_0)|} - \frac{\tilde\gamma(t)-\tilde\gamma(t_0)}{|\tilde\gamma(t)-\tilde\gamma(t_0)|}\right|=0.
\end{equation*}
Thus for any $\epsilon>0$ we can build the point $x\in \support \gamma$ of definition \ref{tangentnormal} that realizes the vector $\tilde\gamma'(t_0)$ as a vector of $\Tan(A,a).$ 

Moreover for a.e. $s_0$ in $]0,1[$ the arc length $t_0=t(s_0):=\Var[\gamma_{[0,s_0]}]$ is an element of $]0,\Var[\gamma][\setminus \Sigma_{\tilde\gamma}$ and $\frac{\gamma'}{|\gamma'|}(s_0)=\tilde\gamma'(t_0).$

Now we recall \cite{Fed59} that since $A$ has positive reach and $\gamma(s_0)\in \partial A$ then  $\Nor(A,\gamma(s_0))$ is not $\{0\}.$  Therefore $\exists v_0\neq 0$ in $\R^d$ such that $\langle \gamma'(s_0),v_0\rangle\leq 0.$

Now we can consider $\bar\gamma(s):=\gamma(1-s)$ and $\bar s_0:=1-s_0$ and apply the same reasoning above to get $$0\leq \langle -\gamma'(s_0),v_0\rangle=\langle \bar\gamma'(\bar s_0),v_0\rangle\leq 0.\;\Rightarrow \gamma'(s_0)\in\langle v_0\rangle^\perp.$$
Taking $v=\frac{v_0}{|v_0|}$ we are done.  
\Square\end{proof}

\section{(Oriented) distance function and $\mathscr C^{1,1}$  domains}
Now we switch to the case of a bounded $\mathscr C^{1,1}$  domain in $\R^d$. 
For the reader's convenience we clarify that here we are using the following definition, however several (essentially equivalent) variants are available.
\begin{definition}\label{c11domain} 
Let $\Omega\subset\R^d$ be a domain, then it is said to be a $\mathscr C^{1,1}$ domain iff the following holds.

There exist $r>0$, $L>0$ such that for any $x\in \partial \Omega$ there exist a coordinate rotation $R_x\in SO^d$ and $f_x\in \mathscr C^{1,1}\left(B^{d-1}(0,r),]-r,r[ \right)$ (that is, a differentiable function having Lipschitz gradient) such that
\begin{eqnarray*}
f_x(0)&=&0\\
\nabla f_x (0)&=&0\\
 \|f_x\|_{\mathscr C^{1,1}}&\leq& L\\
x+R_x\Graph{f_x}&=&\partial\Omega\cap (x+R_xB(x,r)),
\end{eqnarray*}
where$ \|f_x\|_{\mathscr C^{1,1}}:=\max\{\sup_D|f|,\sup_D|\nabla f|,\Lip(\nabla f)\}.$
\end{definition}

In the spirit of \cite{Del} and \cite{zodeldist} one may study regularity properties of a domain $\Omega$ comparing it to the smoothness of the \emph{Distance Function} and the \emph{Oriented Distance Function} $$b_\Omega(\cdot):=d_\Omega(\cdot)-d_{\complement\Omega}.$$
We collect all the properties we need of a $\mathscr C^{1,1}$  domain in $\R^d$ in the following theorem. Detailed proofs can be easily provided combining classical results that can be found in \cite{Barb09}[Th. 5.1.9],\cite{Fed59},\cite{ABMM11} and \cite{zodeldist}.

\begin{theorem}\label{regularirtyoforienteddistance}
 Let $\Omega\subset\R^d$ be a $\mathscr C^{1,1}$  bounded domain. Then the following hold.
\begin{enumerate}[(i)]
 \item Both $\Omega$ and $\complement \Omega$ have positive reach, 
$$\Reach(\partial\Omega)=\min\{\Reach(\Omega),\Reach(\complement\Omega)\}.$$
\item For any $0<h<\Reach(\partial\Omega)$ $b_\Omega\in \mathscr C^1\left(U_h(\Omega)\right)$ where $U_h(\Omega):=\{x\in \R^d:-h<b_\Omega(x)<h \}.$
\item For any $x\in U_h(\Omega),$ $0<h<\Reach(\partial\Omega)$
\begin{equation}
 \nabla b_\Omega(x)=-\frac{x-\pr_{\partial\Omega}(x)}{|x-\pr_{\partial\Omega}(x)|},\label{orienteddistancegradient}
\end{equation}
where the right side is well defined also on $\partial \Omega$. Moreover $\nabla b_\Omega$ is a Lipschitz function.
\item For any $x\in \partial \Omega$ we have $\Tan(x,\partial\Omega)=\mathcal T_x\partial\Omega$ and\\ $\Nor(x,\Omega)=\langle\nabla b_\Omega(x)\rangle.$
\item For all $x\in \partial \Omega$ an d for any $r<\Reach(\partial\Omega)$ we have
\begin{eqnarray}
 B(x-r\nabla b_\Omega(x),r)&\subseteq& \Omega\\
 B(x+r\nabla b_\Omega(x),r)&\subseteq& \complement\Omega\label{uibcc11}
\end{eqnarray}
\end{enumerate}
\end{theorem}

\bibliographystyle{abbrv}
\bibliography{references}
\end{document}